\documentclass[12pt]{article} 
\usepackage{verbatim}
\usepackage{amsmath,amsthm,amssymb,amsfonts,mathrsfs,color}
\usepackage{latexsym,esint}
\usepackage{pstricks}
\theoremstyle{plain}
\begingroup
\newtheorem{thm}{Theorem}[section]
\newtheorem{lem}[thm]{Lemma}
\newtheorem{prop}[thm]{Proposition}
\newtheorem{defin}[thm]{Definition}

\endgroup

\usepackage{graphicx}

\theoremstyle{definition}
\begingroup
\newtheorem{rem}[thm]{Remark}
\endgroup

\numberwithin{equation}{section}

\makeatletter
\def\rightharpoonupfill@{\arrowfill@\relbar\relbar\rightharpoonup}
\newcommand{\xrightharpoonup}[2][]{\ext@arrow 0359\rightharpoonupfill@{#1}{#2}}
\makeatother

\newcommand{\R}{\mathbb R}

\newcommand{\dist}{{\rm dist}}
\newcommand{\diam}{{\rm diam}}

\begin{document}
 
\title{Rectifiability of non Euclidean planar self-contracted curves}
\author{Antoine Lemenant}





\maketitle

\begin{abstract} 
We prove that any self-contracted curve in $\R^2$ endowed with a $C^2$ and strictly convex norm, has finite length. The proof follows from the study of the curve bisector of two points in $\R^2$ for a general norm together with an adaptation of the argument used in \cite{sc}.


 \end{abstract}


\section{Introduction}
 The concept of self-contracted curve was first introduced by Daniilidis, Ley and Sabourau in \cite[Definition
1.2]{DLS}. For a given metric space $(E,d)$ and  a possibly
unbounded interval $I$ of ${\mathbb{R}}$, a map
$\gamma:I\rightarrow E$ is called 
a self-contracted curve, if for every $[a,b]\subset I$, the
real-valued function
\[
t \mapsto d(\gamma(t),\gamma(b))
\]
is non-increasing on $[a,b]$. Sometimes to emphasis on the used distance we will also say $d$-self-contracted curve or $\|\cdot\|$-self-contracted curve when the distance comes from a norm.

The origin of this definition comes from the fact that, in the Euclidean space  $(\R^n,|\cdot |)$, any solution  of the gradient descent of a proper convex fonction $f:{\mathbb{R}}^{n}\rightarrow{\mathbb{R}}$, i.e. a solution $\gamma$ of the gradient system
\begin{equation}
\left\{
\begin{array}
[c]{l}
\gamma^{\prime}(t)=-\nabla f(\gamma(t))\qquad t>0,\\[4pt]
\gamma(0)=x_{0}\in{\mathbb{R}}^{n}
\end{array}
\right.  \label{diffeq}
\end{equation}
 is a  self-contracted curve (actually it is enough for $f$ to be quasiconvex, that is, its sublevel sets to be convex).

One of the main question about solutions   of \eqref{diffeq}, is whether or not
bounded solutions are of finite length. If $f$ is analytic (and not necessarily convex) then it follows from the famous  \L ojasiewicz  inequality \cite{Loja63}, while it fails for  general $C^\infty$ functions  \cite[p.~12]{PalDem82}. Now for a  general convex function,   the \L ojasiewicz  inequality does not need to hold (see \cite[Section~4.3]{BDLM}). However, bounded solutions have finite length.

The latter follows from  the main result of   \cite{DLS} (see also \cite{MP}), which focuses on the following more general question, which is purely metric:
\begin{eqnarray} 
\textit{Does any bounded self-contracted curve have finite length?\label{question1}}
\end{eqnarray}

In   \cite{DLS} it is proved that the answer is yes in a two dimensional Euclidean space.   It was then established in higher dimensions in \cite{sc}, still in the  Euclidean setting. However, it is no more true in a general infinite dimensional Hilbert space.

A natural question is whether the answer to \eqref{question1} remains yes in a finite dimensional space,  when the metric space is not Euclidean anymore. In the recent paper \cite{scR}, it is proved that it holds true on a Riemannian manifold.

In  this paper we prove that  the answer to \eqref{question1} is yes in $\R^2$ endowed with a strictly convex $C^2$ norm.

\begin{thm} \label{mainth}Let  $\|\cdot \|$ be a $C^2$ and strictly convex norm on $\R^2$, let $I\subset \R$ be an interval  and let $\gamma : I\to \R^2$ be a    $\|\cdot\|$-self-contracted curve. Then there exists a  constant $C>0$ depending only on $\|\cdot \|$ such that 
\begin{eqnarray}
\ell(\gamma)\leq C\diam(K(\gamma)) \label{length}
\end{eqnarray}
where  $\ell(\gamma)$ is the length of the curve, $K(\gamma)$ is the closed convex hull of the support of the curve $\gamma$, and $\diam(K)$ is the diameter of $K$. In particular, any bounded $\|\cdot\|$-self-contracted curve has finite length.\end{thm}

Notice that to establish the inequality $\eqref{length}$, the norm which is used to compute $\ell(\gamma)$ or $\diam(K)$ does not really matter, since all the norms are equivalent on $\R^2$. Actually, we will  establish \eqref{length} with $\ell(\gamma)$ and $\diam(K)$ computed using the Euclidean norm on $\R^2$, denoted $|\cdot |$. More precisely, our strategy is to reproduce the Euclidean argument used in \cite{sc}, in the simple case of dimension 2, and prove that if $\gamma$ is no more $|\cdot|$-self-contracted but merely $\|\cdot \|$-self-contracted, then the proof  still works almost  by the same way.

However, there are some notable differences. The main one concerns the starting point of the proof, which consists, for any $t_0<t_1$, of localizing  the curve  ``after'' $t_1$ on one side of the perpendicular bisector  of the segment $[\gamma(t_0),\gamma(t_1)]$. This   follows directly from the definition of being self-contracted.  Indeed, for $t_0<t_1<t$, since the function   $s\mapsto d(\gamma(s),\gamma(t))$  is non-decreasing on $[t_0,t]$ we have  
$$d(\gamma(t_0),\gamma(t))\leq d(\gamma(t_1),\gamma(t))$$
which means that $\gamma(t)$ is situated on one side of the bisector 
$$M(\gamma(t_0),\gamma(t_1)):=\{x \in \R^2\; : \; d(\gamma(t_0),x)= d(\gamma(t_1),x)\}.$$

 For the Euclidean distance this yields a ``separating line'' for the curve after $t_0$, whose direction is orthogonal to $\gamma(t_0)-\gamma(t_1)$. If the distance is not Euclidean anymore, the line segment bisector  is no longer perpendicular, and not even   a line anymore. However, if the distance is coming from a $C^2$ and strictly convex norm, we prove that the   bisector is a curve which is asymptotic to a line at infinity, whose direction is well identified: in a certain sense it is a direction  dual  to $\gamma(t_0)-\gamma(t_1)$. Moreover, the bisector  stays close enough to the middle line having the same direction. After noticing those facts, we are able to adapt the proof of \cite{sc} and this is how we prove Theorem \ref{mainth}.

\medskip
 {\bf Acknowledgements:} The author wishes to thank Aris Daniilidis, Estibalitz Durand-Cartagena,  Michael Goldman  and Vincent Munnier for useful discussions on the subject of this paper. This work was partially supported by the project  ANR-12-BS01-0014-01 GEOMETRYA  financed by the French Agence Nationale de la Recherche (ANR).

\subsection{Notation and terminology} In this paper we will work on $\R^2$ endowed with the Euclidean norm $|x|$ with scalar product $\langle x,y\rangle$. The Euclidean ball with center $x$ and radius $R$ will be denoted by $B(x,R)$ and  $\mathbb{S}^{N-1}$ is the Euclidean unit sphere. For $v\in \R^2$ we will denote by $v^\bot$ the image of $v$ the rotation of angle $\frac{\pi}{2}$ in the anticlockwise direction.

Eventually, on this Euclidean space $\R^2$ we shall also consider another non Euclidean norm that will be denoted by $\|\cdot\|$. The associated ball will be denoted by $B_{\|\cdot \|}(x,r)$ and the sphere for the norm $\|\cdot \|$ will then be denoted by $\partial B_{\|\cdot \|}(x,r)$.


A curve is a mapping   $\gamma:I \to \R^2$, not necessarily continuous,  from  some interval $I\subset \R$. The length of a curve is the quantity 
$$\ell(\gamma):=\sup\left\{ \sum_{i=0}^{m-1} |\gamma(t_i)-\gamma(t_{i+1})|\right\},$$
where the supremum is taken over all finite increasing sequences $t_0<t_1<\dots<t_m$ that lie in the interval $I$.

The mean width of a convex $K\subset \R^2$ is the quantity 
$$W(K)=\frac{1}{2\pi}\int_{\mathbb{S}^1}\mathcal{H}^1(P_u(K))du$$
where $P_u$ is the orthogonal projection onto the real line in $\R^2$ directed by the vector $u\in \mathbb{S}^1$, and $\mathcal{H}^1$ denotes  the 1-dimensional Hausdorff measure. Although we shall not use it in this paper, let us mention the following nice identity valid   for any compact and  convex set $K\subset \R^2$,
$$W(K)=\frac{\mathcal{H}^{1}(\partial K)}{\pi}.$$
 It is also clear from the definition that 
$$W(K)\leq \diam(K).$$

If $x,y\in \R^2$ we will denote by $[x,y]\subset \R^2$ the segment between $x$ and $y$.

A norm $\|\cdot \|$ is said to be $C^2$ if $x\mapsto \|x \|$ is of class $C^2$ on $\R^2\setminus \{0\}$. This implies that the sphere  $\partial B_{\|\cdot \|}(x,r)$ is a $C^2$ manifold.
A norm is said to be $C^2$ and strictly convex if, for any $x\in \R^2\setminus \{0\}$ we have $D^2\|x\| >0 $ (in the sense of quadratic forms). This implies that the ball $B_{\|\cdot\|}(0,1)$ is strictly convex in the geometrical sense, i.e. for any couple of points $x, y \in  \partial B_{\|\cdot \|}(0,1)$ we have $[x, y] \cap   \partial B_{\|\cdot \|}(0,1)=\{x,y\}$. In other words the segment $[x,y]$ without its extremities $x$ and $y$ lies in the interior of the ball $B_{\|\cdot \|}(0,1)$. This means that  the sphere  $\partial B_{\|\cdot \|}(0,1)=\{x,y\}$ contains no ``flat-parts''.

\section{Preliminaries about the curve bisector}
 Let $\|\cdot \|$ be a $C^2$ and strictly convex norm on $\R^2$, and $a,b\in \R^2$ be any given points. Then we focus on the curve bisector defined by 
 $$M(a,b):=\{z\in \R^2 \; : \; \|a-z\|=\|b-z\|\}.$$

If the norm is  Euclidean, it is well known that $M(a,b)$ is the line passing through the middle point $\frac{a+b}{2}$ and  perpendicular  to $b-a$. But for a general norm, the curve bisector may not be perpendicular and may even not be a line. However, it is not difficult to see  that $M(a,b)$ is always asymptotically converging to a line at infinity, and the direction of that line is given by the  dual vector to $b-a$ (i.e. the vector on the sphere $B_{\|\cdot\|(0,1)}$ at which the tangent to the sphere has  direction $b-a$).  Figure \ref{picture3} represents the curve bisector of two points for the norm $\|\cdot \|_4$ in $\R^2$.

\begin{figure}[htbp]
\begin{center}
\includegraphics[width=11cm]{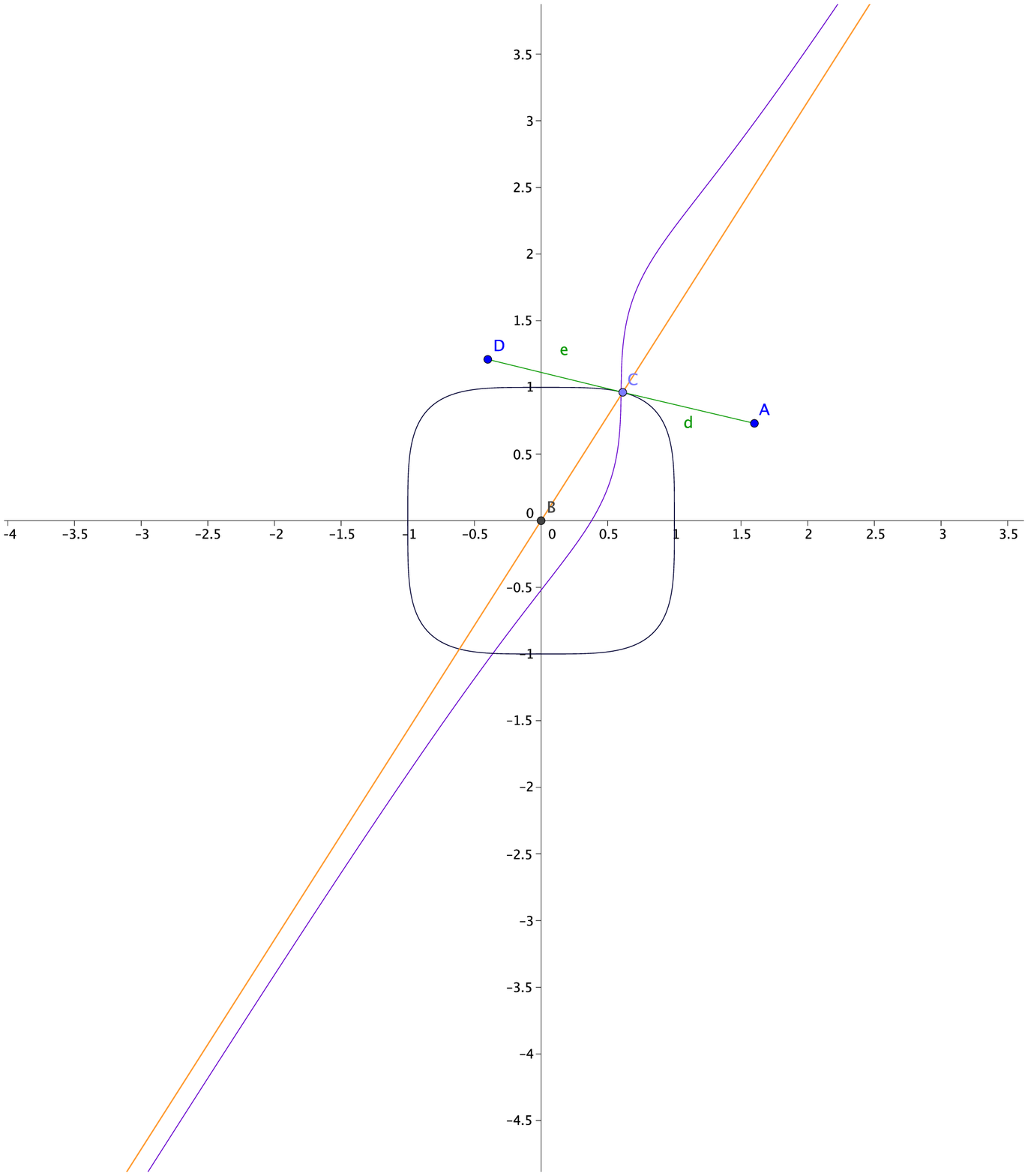}
\end{center}
\caption{The curve bisector of the segment $[A,D]$ for the norm $\|\cdot \|_4$ is a  curve which is asymptotic to a line at infinity, whose direction is given by the point on the sphere $\partial B_{\|\cdot \|_4}(0,1)$ at which the tangent to the sphere has the same direction as the  segment [A,D]. } \label{picture3}
\end{figure}

To prove this, we shall need to consider the direction defined in the following definition.


\begin{defin}[Definition of $L_x$]Let $\|\cdot \|$ be a $C^2$ and strictly convex norm on $\R^2$. For any $x \in \partial B_{\|\cdot\|}(0,1)$ we define the vectorial line ``dual to $x$'', denoted $L_x$, in the following way: there exists exactly two points $y_1,y_2\in \partial B_{\|\cdot\|}(0,1)$ such that the tangent line to $ \partial B_{\|\cdot\|}(0,1)$ at point $y_i$ is directed by $x$. Then $L_x$ is the line passing through $y_1$ and $y_2$. If $x\not  =0$ and $\|x\|\not = 1$, we denote by $L_x$ the line $L_{x/\|x\|}$.
\end{defin}

\begin{rem}\label{defalpha0} A geometrical way to find the $y_i$ is as follows. Let $H_x\subset \R^2\setminus B_{\|\cdot\|}(0,1)$ be a line directed by $x$. Then translate this line  in the direction $x^\bot$ until it touches the sphere $ \partial B_{\|\cdot\|}(0,1)$. The touching point is then one of the $y_i$. Since the norm is  $C^2$ and strictly convex, there exists only two of those points, which from the symmetry of the ball with respect to the origin must satisfy $y_1=-y_2$. For example, in Figure 1, the line passing through the origin is $L_{D-A}=L_{A-D}$.
\end{rem}

\begin{rem} \label{defalpha0} For all $x \in  \partial B_{\|\cdot\|}(0,1)$, let $\alpha(x)\in [0,\frac{\pi}{2}]$ be the smallest angle between the two lines $L_x$ and $\R x$. Then it is easy to see that $\frac{\pi}{2}\geq \inf_{x \in   \partial B_{\|\cdot\|}(0,1)} \alpha(x)=:\alpha_0>0$.  Indeed, the infimum is actually a minimum, and it cannot be zero because the origin lies in the interior of the ball.  Roughly speaking,  the angle $\alpha_0$ quantifies how far the ball for $\|\cdot\|$ is from being the Euclidean ball, in the sense of how  the tangent line to the sphere can be far from making an angle of   $\frac{\pi}{2}$ with the radius. Another way to define $\alpha_0$ is as follows: for every $x\in  \partial B_{\|\cdot\|}(0,1)$ let $\nu_x$ be  outer unit normal vector to $\partial B_{\|\cdot\|}(0,1)$ at point $x$. Then $\alpha_0$ is the unique real number in $[0,\frac{\pi}{2}]$ solution to  
\begin{eqnarray}
\sin (\alpha_0)=\inf_{x\in \partial B_{\|\cdot\|}(0,1) } \langle  \nu_x, \frac{x}{|x|}\rangle .\label{defalpha02}
\end{eqnarray}
\end{rem}

The following proposition says that $M(a,b)$ converges to the line 
\begin{eqnarray}
L(a,b):=\frac{b+a}{2}+L_{b-a} \label{orthline}
\end{eqnarray} 
at infinity. 

\begin{prop} \label{firstprop} Let $\|\cdot \|$ be a $C^2$ and strictly convex norm on $\R^2$. Then for every $\varepsilon>0$ there exists $R>0$ such that for all couple of points $a,b \in \R^2$ satisfying $a\not = b$ and $\|b-a\|=1$, 
$$\sup \{\dist(z,L(a,b)) \; : \; z\in M(a,b)\setminus B(a,R) \}\leq \varepsilon,$$
where $L(a,b)$ is defined in \eqref{orthline}, and $\dist$ is the Euclidean distance.\end{prop}

\begin{proof} Let $a,b \in \R^2$ be satisfying $a\not = b$ and $\|b-a\|=1$, and let  $z_R \in M(a,b)$ be such that $\|z_R-a\|=R\to +\infty$. Assume for simplicity that $z_R$ stays ``above'' the line directed by $b-a$. Then  it is easy to see that $\frac{z_R-a}{R}$ and $\frac{z_R-b}{R}$ both converge to the same direction $w\in \partial B_{\|\cdot \|}(0,1)$, the one directing the line $L_{b-a}$.  To see this, we let  $R_n$ be a subsequence making $w_n:=\frac{z_{R_n}-a}{R_n}$ and  $w_n':=\frac{z_{R_n}-b}{R_n}$ converging on the compact set $\partial B_{\|\cdot \|}(0,1)$ to some $w$ and $w'$.  We actually have that $w=w'$ because
$$\|w_n-w'_n\|=\|b-a\|/R_n=1/R_n\to 0.$$
 Now we claim that $w$ is directing the line $L_{x}$, where $x=b-a$. To see this, it is enough showing that the tangent line to $ \partial B_{\|\cdot\|}(0,1)$ at point $w$ is directed by $x$. This follows from the fact that, since 
 $ \partial B_{\|\cdot\|}(0,1)$ is a $C^2$ manifold, the rescaled secants $(w_n-w_n')/\|w_n-w_n'\|$, which is nothing but $a-b$,  converges to the tangent line  to $\partial B_{\|\cdot\|}(0,1)$ at the point $w$.

 Now to prove the proposition, we need to go one order further, namely use  a $C^2$ argument. For this purpose, we adopt the following point of view: let  $a,b \in \R^2$ be such that $\|a-b\|=1$ and let $v=b-a$ the direction of the segment $[a,b]$. For all $t\in \R$, we denote by $H_t:=(0,t)+\R v$ the line  directed by $v$ in $\R^2$, of height $t$. Let   $t_0>0$ be the maximum  $t$ such that  $H_t\cap B_{\|\cdot\|(0,1)}\not = \emptyset$. Let also $a_t$ and $b_t$ be the two points of $H_t\cap \partial B_{\|\cdot\|}(0,1)$, for $t\in(-t_0,t_0)$. Then for any $t\in (-t_0,t_0)$, it is clear that the origin  lies on the curve bisector $M(a_t,b_t)$ (see Figure \ref{picture4}). Actually,  by using the rescaling $(b_t-a_t)/\|b_t-a_t\|$  we get this way a  complete parametrization on $(-t_0,t_0)$ of $M(0,v)$, the curve bisector of a the segment of unit length directed by $v$ starting from the origin (which is the same as $M(a,b)$, after translation to the origin).

Moreover, when $t$ converges to $-t_0$ or  $t_0$, the points  $a_t$ and $b_t$ converge to some points $-y$ and $y$ on the sphere $\partial B_{\|\cdot\|}(0,1)$, the ones at  which the tangent line to $\partial B_{\|\cdot\|}(0,1)$ is directed by $v$. The line passing through $-y$ and $y$ is exactly $L_{v}$.

Next, for any $R\geq \frac{1}{2}$, there exists a unique $t_R\in (-t_0,0)$ such that $\frac{1}{\|a_{t_R}-b_{t_R}\|}=R$. Therefore, if $z_R \in M(0,v)$ is such that $\|z_R\|=R$, 
\begin{eqnarray}
dist(z_R, L(0,v))\leq R|m_{t_R}-l_{t_R}|, \label{estii}
\end{eqnarray}
where $m_{t}=(a_{t}+b_{t})/2$ is the middle point of the segment $[a_t,b_t]$, and $l_t$ is the intersecting point of the two segments $[0,-y]$ and $[a_t,b_t]$ (see Figure \ref{picture4}).

\begin{figure}[htbp]
\begin{center}
{
\begin{pspicture}(0,-3.2303808)(12.106958,3.2303808)
\rput{-45.0}(1.6970041,4.5161686){\psellipse[linecolor=black, linewidth=0.048, dimen=outer](6.3,0.20961925)(2.3,3.6)}
\psdots[linecolor=black, dotsize=0.12](6.4,0.20961925)
\psline[linecolor=black, linewidth=0.048](0.4,-0.99038076)(11.0,-0.99038076)
\psline[linecolor=black, linewidth=0.03](0.0,-2.7903807)(11.2,-2.7903807)
\psline[linecolor=black, linewidth=0.03](6.4,0.20961925)(5.0,-2.7903807)
\psdots[linecolor=black, dotsize=0.12](3.3,-0.99038076)
\psdots[linecolor=black, dotsize=0.12](8.28,-0.99038076)
\rput[bl](6.36,0.62961924){\normalsize{$0$}}
\psline[linecolor=black, linewidth=0.02](3.32,-0.9703807)(6.4,0.20961925)
\psline[linecolor=black, linewidth=0.02](6.4,0.20961925)(8.26,-0.9703807)
\rput[bl](11.26,-1.0903808){\normalsize{$H_t$}}
\rput[bl](2.56,-1.4303807){\normalsize{$a_t$}}
\rput[bl](8.2,-1.4303807){\normalsize{$b_t$}}
\psdots[linecolor=black, dotsize=0.12](5.58,-0.99038076)
\rput[bl](5.233333,-0.8037141){\normalsize{$m_t$}}
\psdots[linecolor=black, dotsize=0.12](5.02,-2.7903807)
\rput[bl](4.48,-3.2303808){\normalsize{$-y$}}
\rput[bl](4.98,-0.010380747){\normalsize{$1$}}
\rput[bl](7.4,-0.35038075){\normalsize{$1$}}
\rput[bl](5.82,2.2496192){\normalsize{$B_{\|\cdot\|}(0,1)$}}
\psline[linecolor=black, linewidth=0.04, tbarsize=0.07055555cm 5.0,arrowsize=0.05291667cm 2.0,arrowlength=1.4,arrowinset=0.4]{|->}(9.8,0.40961924)(12.2,0.40961924)
\rput[bl](10.4,0.60961926){\normalsize{$v$}}
\psdots[linecolor=black, dotsize=0.12](5.84,-0.99038076)
\rput[bl](6.1466665,-0.85704744){$l_t$}
\end{pspicture}
}
\end{center}
\caption{The distance from $M(a_t,b_t)$ to $L(a,b)$ is the same as the distance from $m_t$ to  $L(b-a)$ (the line passing through $0$ and $-y$), which is less than the distance between $m_t$ and $l_t$.}
 \label{picture4}
\end{figure}
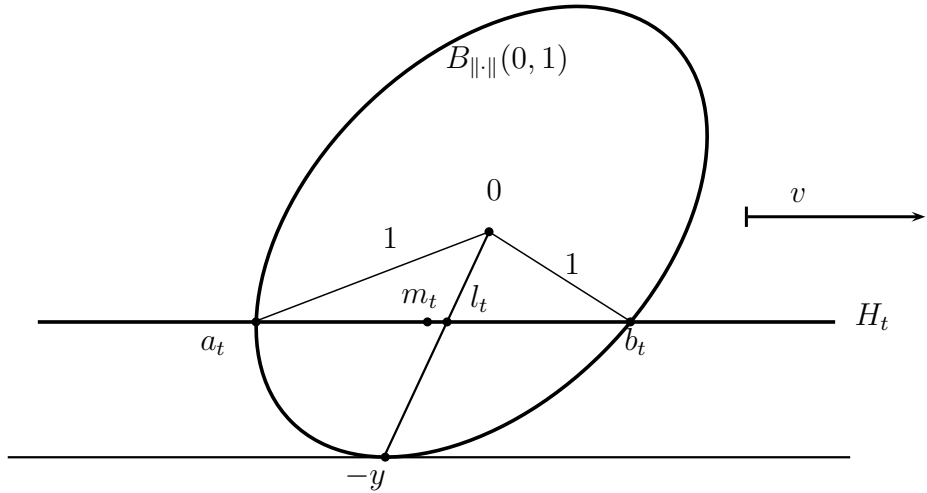

Now, we claim that, when $|a_{t}-b_{t}|\to 0$, 
\begin{eqnarray}
|m_{t}-l_t|\leq C|a_{t}-b_{t}|^2+o(|a_t-b_t|^2), \label{mtr}
\end{eqnarray}
for some $C>0$ (depending only on $\|\cdot\|$). This is enough to conclude because returning to \eqref{estii} we get, using that $\|a_{t_R}-b_{t_R}\|=\frac{1}{R}$,
$$dist(z_R, L(0,v))\leq R|m_{t_R}-l_{t_R}|\leq C'\frac{1}{R}\underset{R\to +\infty}{\longrightarrow} 0.$$

Now to prove \eqref{mtr}, we work locally around the point $-y \in \partial B_{\|\cdot\|}(0,1)$. Let $T_{-y}$ be the tangent line at point $-y$. Since the norm is $C^2$ and strictly convex, in a certain well chosen coordinate system, the sphere  $\partial B_{\|\cdot\|}(0,1)$  almost  coïncides with the graph of $s\mapsto As^2$ for some $A>0$,  up to some error of second order.  Consequently,  in that coordinate system, the segments $[a_t-b_t]$ which have the same direction as $T_{-y}$, almost coïncides with the segment $[-x_s,x_s]$ where $x_s=(s,As^2)$, again up to some error of second order.  On the other hand the angle $\theta$ between the radius $[0,-y]$  of the ball $B_{\|\cdot\|}(0,1)$ and the  vertical line orthogonal to $T_{-y}$, is  at most $\frac{\pi}{2}-\alpha_0$, where $\alpha_0$  is defined in Remark \ref{defalpha0} (see Figure \ref{picture5}).

\begin{figure}[htbp] 
\begin{center}
%
%
\psscalebox{1.0 1.0} 
{
\begin{pspicture}(0,-3.365909)(16.363636,3.365909)
\definecolor{colour0}{rgb}{0.6,0.6,0.6}
\definecolor{colour1}{rgb}{0.4,0.4,0.4}
\psline[linecolor=black, linewidth=0.04](0.0,-2.9113636)(16.363636,-2.9113636)
\psarc[linecolor=black, linewidth=0.04, dimen=outer](8.309091,2.9977272){5.909091}{200.13608}{338.1978}
\psdots[linecolor=black, dotsize=0.12](8.236363,-2.8931818)
\rput[bl](8.109091,-3.365909){$-y$}
\psline[linecolor=black, linewidth=0.04](8.236363,-2.9477272)(8.236363,2.3795455)
\psline[linecolor=black, linewidth=0.04](8.2,-2.875)(6.3454547,3.125)
\rput[bl](6.6363635,3.0159092){$[0,-y]$}
\rput[bl](14.4,-2.7659092){$T_{-y}$}
\psarc[linecolor=black, linewidth=0.04, dimen=outer](8.245455,-3.175){2.6818182}{89.95318}{105.477295}
\rput[bl](8.745455,0.37954545){$\theta\leq \frac{\pi}{2}-\alpha_0$}
\psline[linecolor=black, linewidth=0.04, arrowsize=0.05291667cm 2.0,arrowlength=1.4,arrowinset=0.0]{<-}(7.981818,-0.49318182)(8.709091,0.37954545)
\psline[linecolor=black, linewidth=0.04](3.3090909,-1.9840909)(15.0,-1.9840909)
\rput[bl](8.290909,-2.6022727){$As^2$}
\rput[bl](4.7454543,-1.8022727){$-x_s$}
\rput[bl](11.345454,-1.7840909){$x_s$}
\psline[linecolor=black, linewidth=0.04, linestyle=dashed, dash=0.17638889cm 0.10583334cm](11.436363,-2.0204546)(11.418181,-2.8931818)
\rput[bl](11.363636,-3.3295455){$s$}
\psbezier[linecolor=colour0, linewidth=0.04](8.2,-2.9113636)(3.9818182,-2.7766464)(2.6181817,-1.3402828)(1.7454545,0.16136363636363968)
\psbezier[linecolor=colour0, linewidth=0.04](8.254545,-2.9113636)(12.090909,-2.875)(13.945455,-0.5840909)(14.654546,0.470454545454545)
\rput[bl](14.236363,0.48863637){\textcolor{colour1}{$\partial B_{\|\cdot\|}(0,1)$}}
\psdots[linecolor=black, dotsize=0.12](5.1272726,-1.9840909)
\psdots[linecolor=black, dotsize=0.12](11.418181,-1.9840909)
\psdots[linecolor=black, dotsize=0.12](12.0,-1.9659091)
\psdots[linecolor=black, dotsize=0.12](3.8727272,-1.9840909)
\rput[bl](3.5636363,-2.293182){$a_t$}
\rput[bl](12.036364,-2.365909){$b_t$}
\end{pspicture}
}
\end{center}
\caption{Situation in the proof of \eqref{mtr}.}
\label{picture5}
\end{figure}
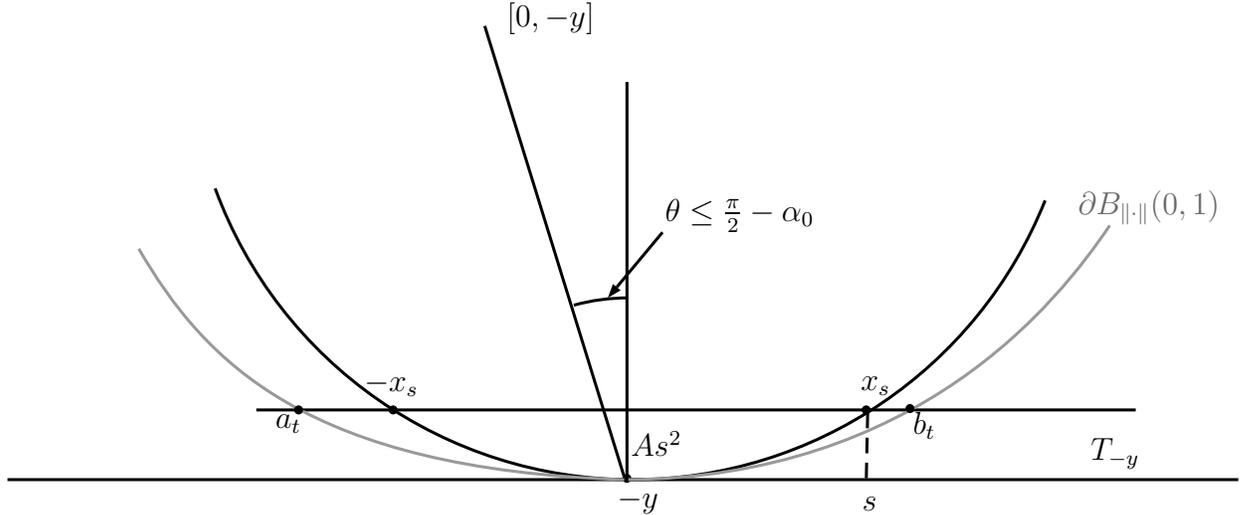

We deduce that 
$|m_{t}-l_t|\leq |A\sin(\theta)s^2|+o(s^2)=Cs^2+o(s^2),$
as desired. This achieves the proof of \eqref{mtr}, and of the proposition.

\end{proof}

\begin{defin}For any $x\in \R^2$ and $v\in \R^2$  with $\|v\|=1$ we denote by $Q_v(x)$ the projection of $x\in \R^2$ on $\R v$  parallely to $L_v$ (i.e. the $v$ component of $x$ written in the basis $(v,v')$ where $v'$ is a vector director for $L_v$). 
\end{defin}

The next proposition quantifies the maximum distance between $M(a,b)$ and the line $L(a,b)$. 

\begin{prop}\label{prop2} Let $\|\cdot \|$ be a $C^1$ and strictly convex norm on $\R^2$.  There exists $0<\kappa <1/2$ depending on the norm $\|\cdot \|$, such that for any given $a,b \in \R^2$, the curve bisector $M(a,b)$ is contained in the strip $S_\kappa(a,b)$ defined as 
$$S_{\kappa}(a,b):=\{x\in \R^2 \; : \;   |Q_{\frac{b-a}{\|b-a\|}}(x-\frac{a+b}{2})|\leq \kappa \|b-a\| \}.$$
\end{prop}
\begin{proof} We first prove that for all $a$, $b$ in $\R^2$, $M(a,b)\subseteq S_{1/2}(a,b)$. Let $a,b$ be given. We can assume that $\|b-a\|=1$ and we denote by $v=b-a$. Adopting the same notation  $a_t, b_t, m_t$ as    in the proof of  Proposition \ref{firstprop},  we see that 
$$\sup \{|Q_v(x-v/2)| \; :  \; x \in M(0,v) \}= \sup_{t\in[-t_0,t_0]} \frac{|Q_v(m_t)|}{\|b_t-a_t\|}.$$
By convexity of the ball $B_{\|\cdot\|(0,1)}$ we have  $L_v\cap [a_t,b_t]\not =\emptyset$ for all $|t|\leq t_0$, from which we find that
$$ \sup_{t\in[-t_0,t_0]} \frac{|Q_v(m_t)|}{\|b_t-a_t\|}\leq 1/2,$$
 as claimed.

Now to pass from $1/2$ to some $\kappa<1/2$ we argue by contradiction. Assuming  the proposition is false, we can find $a_n:=a_{t_n}$, $b_n:=b_{t_n}$,   such that 
\begin{eqnarray}
\frac{|Q_v(m_n)|}{\|b_n-a_n\| }\underset{n\to +\infty}{\longrightarrow} 1/2, \label{essai2}
\end{eqnarray}
where $m_n:=(a_n+b_n)/2$. According to Proposition \ref{firstprop}, we can assume that 
\begin{eqnarray}
\inf_n \|a_n-b_n\|\geq \delta, \label{lowbound}
\end{eqnarray}
for some $\delta>0$. Indeed, let $R>0$ be given by Proposition \ref{firstprop} for  $\varepsilon:=1/4$ and let  $\delta:= 1/R$. Then, for all $a_t$ and $b_t$ satifying $r:=\|a_t-b_t\|\leq \delta$, we  have that $\|a_t/r\|=1/r\geq 1/\delta = R$ thus $(0,0)\in M(a_t/r,b_t/r)\setminus B(a_t/r,R)$ and therefore the conclusion of Proposition \ref{firstprop} says 
$$\dist((0,0),L(a_t/r,b_t/r))\leq  \varepsilon=\frac{1}{4}.$$
In particular this implies $|Q_v(\frac{a_t+b_t}{2})|/\|b_t-a_t\|\leq 1/4$. Hence, our sequence of $a_n$ and $b_n$ satisfying  \eqref{essai2} must satisfy  $\|a_n-b_n\|\geq \delta$ for $n$ large enough. Now by compactness of $\partial B_{\|\cdot\|}(0,1)$ we can assume that $a_n\to a= a_{t_0}$ and $b_n\to b=b_{t_0}$ satisfying $\|a-b\|\geq \delta$. By \eqref{essai2} we also have that 
\begin{eqnarray}
\frac{|Q_v(m)|}{\|b-a\| }= 1/2 ,\notag
\end{eqnarray}
where $m=(a+b)/2$. But this implies that $L_v\cap [a,b]$ is either the point $a$ or the point $b$, which is a contradiction because, by strict convexity of the ball $B$, the line $L_v$ must meet the segment $[a,b]$ only by its interior. This finishes the proof of the Proposition.
\end{proof}

\section{Rectifiability of $\|\cdot\|$-SC-curves}


In the sequel if $\gamma$ is a $\|\cdot\|$-self-contracted curve, we denote by $\Gamma:=\gamma(I)$ the image of the curve, and for all $x\in \Gamma$  we introduce the ``piece of curve after $x$''  namely,  
$$\Gamma(x):=\{y \in \Gamma \; : \; x\preceq y\},$$
where $\preceq$ denotes the order on the curve given by its parameterization. We will also  denote by $\Omega(x)$ the convexe hull of $\Gamma(x)$.

We start with a first lemma about the maximum aperture of the angle between $y'-x_0$ and $y-x_0$ when $x_0,y,y'$ are all lying on $\Gamma$ with $y$ and $y'$ after $x_0$. Namely, for an Euclidean self-contracted-curve, it is easy to see that whenever $x_0,y,y'\in \Gamma$ are satisfying $x_0\preceq y\preceq y'$ we have 
$$\langle y-x_0 , y'-x_0 \rangle\geq 0.$$
The way this is proved in \cite{sc} is as follows. From the self-contracted property we infer that
$$|y'-x_0|\geq |y'-y|.$$
Thus writing $y'-y=(y'-x_0)+x_0-y$ and squaring the estimates we get
$$|y'-x_0|^2\geq |y'-x_0|^2+|y-x_0|^2 -2\langle y'-x_0,y-x_0\rangle$$
which yields 
$$\langle y'-x_0,y-x_0\rangle\geq 0.$$

For a general norm which is no more Euclidean, we get a similar estimate from a different argument.

\begin{lem}\label{estimateL} Let $\|\cdot\|$ be a $C^2$ and strictly convex norm on $\R^2$. Let $\alpha_0$ be the constant of Remark \ref{defalpha0} and let $\gamma:I\to \R^2$ be a $\|\cdot\|$-self-contracted curve. Then for every $x_0,y,y'\in \Gamma$ satisfying $x_0\preceq y\preceq y'$ we have 
\begin{eqnarray}
\big\langle \frac{y-x_0}{|y-x_0|} , \frac{y'-x_0}{|y'-x_0|} \big\rangle\geq -\cos(\alpha_0). \label{boundcos}
\end{eqnarray}
\end{lem}
\begin{proof} The argument is purely geometric. Let $x_0,y,y'\in \Gamma$ satisfying $x_0\preceq y\preceq y'$. The self contracting property yields
$$\|y'-x_0\|\geq \|y'-y\|$$
which means that $x_0 \not \in B_{\|\cdot \|}(y',\|y-y'\|)$. Hence, assuming   $y'=0$ and $\|y\|=1$, to find a lower bound for the left-hand side of \eqref{boundcos} we can consider the following problem, denoting  $B:=B_{\|\cdot \|}(0,1)$, and fixing $y \in \partial B$,
\begin{eqnarray}
\inf_{x\in \R^2\setminus B} \big\langle \frac{y-x}{|y-x|}, \frac{-x}{|x|}  \big\rangle. \label{aresoudre}
\end{eqnarray}
To solve this problem, we first prove that it is equivalent to take the infimum in $x\in \partial B$. Indeed, let us define
$$\varphi(x):=\langle \frac{x-y}{|y-x|}, \frac{x}{|x|}\rangle$$
and consider the function $f$ of the real variable $t>0$ defined by  $f(t)=\varphi(tx)$. A simple computation shows that 
$\frac{d}{dt}|y-tx|=\frac{\langle tx-y,x \rangle}{|y-tx|}$ and $\frac{d}{dt}\langle tx-y,x\rangle=|x|^2$ thus
$$f'(t)=\frac{|x|^2|y-tx|^2-\langle tx-y,x\rangle^2}{|x||y-tx|^3}\geq 0,$$
due to the Cauchy-Schwarz inequality. This implies that $t\mapsto \varphi(tx)$ is non decreasing in $t$ and the infimum  in \eqref{aresoudre} is the same as: 
\begin{eqnarray}
\inf_{x\in \partial B} \big\langle \frac{x-y}{|y-x|}, \frac{x}{|x|}  \big\rangle. \label{aresoudre1}
\end{eqnarray}
But now it is easy to conclude using the convexity of $B$. Indeed, for a given $x\in \partial B$, let   $\nu_x$ be the unit inner normal  vector  to $\partial B$ at point $x$. Then by convexity of $B$,  since $y \in \partial B$ we must have 
$$\langle y-x, \nu_x\rangle\geq 0,$$
and by definition of $\alpha_0$  we deduce that
$$  \big\langle \frac{x-y}{|y-x|}, \frac{x}{|x|}  \big\rangle\geq -\cos(\alpha_0).$$
 
\end{proof}

We are now ready to prove the main result of this paper.

\begin{proof}[Proof of Theorem \ref{mainth}]  We use the same notation $\prec$, $\Gamma$, $\Gamma(x)$, $\Omega(x)$ as before. The begining of the proof follows essentially the proof of  \cite[Theorem 3.3]{sc} which is divided into several steps. Although the first two steps are very close to the Euclidean situation,  we write here the full detail.

The first step consists in first noticing   that to prove the theorem, it is enough finding some $c_0>0$ such that for any pair of points $x,x' \in \Gamma$ with $x'\preceq x$ it holds
\begin{eqnarray}
W(\Omega(x))+c_0 |x-x'|\leq W(\Omega(x')) \label{estim2}
\end{eqnarray}
(this is Claim 1 of  \cite[Theorem 3.3]{sc}).

Indeed, letting
$t_{0}<t_{1}\ldots<t_{m}$ be any increasing sequence in $I$, and set
$x_{i}:=\gamma(t_{i})$. If \eqref{estim2} holds, then
\[
\begin{aligned} \sum_{i=0}^{m-1}|\gamma(t_{i+1})- \gamma(t_i)|
&= \sum_{i=0}^{m-1}|x_{i+1}- x_{i}| \leq \frac{1}{c_0}
\sum_{i=0}^{m-1}\big(W(\Omega(x_{i})-W(\Omega(x_{i+1}))\big) \cr& =
\frac{1}{c_0} (W(\Omega(x_{0}))-W(\Omega(x_{m}))) \leq
\frac{1}{c_0} W(\Omega(x_{0})),
\end{aligned}
\]
since the mean width $W(H)$ is a nondecreasing function of $H$ (the variable
$H$ is ordered via the set inclusion). Taking the supremum over all choices of
$t_{0}<t_{1}\ldots<t_{m}$ in $I$ we obtain  \eqref{length} for
$C=1/c_0$. 

\medskip

Therefore, the theorem will be proved, if we show that \eqref{estim2} holds
for some  constant $c_0>0$. Before
we proceed, we introduce some extra notation similar to the ones of \cite{sc} but modified with the constant $\kappa$. Indeed, let $\kappa$ be the constant given by Proposition \ref{prop2} depending only on $\|\cdot\|$ and let $\lambda:=\frac{1}{2}-\kappa>0$. Let  $x,x^{\prime}$ be fixed in $\Gamma$
with $x^{\prime}\prec x$. We set (see Figure \ref{picture1})
\begin{equation}
v_{0}:=\frac {x^{\prime}-x}{| x^{\prime}-x|} \quad \text{ and } \quad x_{0}:=x'-v_0 \frac{\lambda}{2}|x-x'|.\label{guy}
\end{equation}

 Let us also set
\begin{equation}
\xi_{0}(y)=\frac{y-x_{0}}{|y-x_{0}|}\in\mathbb{S}^{1}
,\quad\text{for any }y\in\Gamma(x). \label{guy-1}
\end{equation}
Notice that $x_{0}$, $v_{0}$ and $\xi_{0}(y)$ depend on the points $x,x^{\prime
},$ while the desired constant $c_0$ does not (may depend only on $\|\cdot\|$). To determine $c_0$, we shall again transform the problem into another one (similar to  Claim~2 of \cite[Theorem 3.3]{sc}).

\medskip

Let us assume that there exists some constants
$\delta>0$ and $\tau>0$ (depending only on $\|\cdot\|$), such that
for all $x,x^{\prime}$ in $\Gamma$ with $x^{\prime}\prec x$ (and for
$x_{0},v_{0}$ defined by \eqref{guy}), there exists
$\bar{v}\in\mathbb{S}^{1}$ such that the following two properties hold:
\begin{eqnarray}
\langle \bar{v} , v_0 \rangle \geq  \tau \label{CL20}
\end{eqnarray}
\begin{equation}
\big\langle\overline{v},\xi_{0}(y)\big\rangle\leq-\delta\quad\text{
for all }y\in\Gamma(x). \label{CL21}
\end{equation}
Then \eqref{estim2} holds true (and consequently \eqref{length} follows).

This is a slight modification of Claim 2 of  \cite[Theorem 3.3]{sc}).  To prove the latter, assume that such constants $\delta, \tau$, and a
vector $\bar{v}$ exist, so that \eqref{CL20} and \eqref{CL21} holds. Up to change $\delta$ into $\tau/2$ if necessay, we may assume that $\delta\leq \tau/2$. Set
\[
V:=\big\{v\in\mathbb{S}^{1}\,;\, | v-\overline{v} |\leq\delta
\big\}.
\]
From \eqref{CL21} we get
\begin{equation}
\langle v,y-x_{0}\rangle\leq0,\quad\text{for all }v\in V\;\text{and
} y\in\Omega(x). \label{polar}
\end{equation}

Let us now assume that $x_0$ is the origin. Recall that, for $v\in \R^2$, $P_{v}$ denotes the orthogonal projection onto the line $\mathbb{R}v$. Observe that $\Omega(x^{\prime})$ contains the convex hull of
$\Omega(x)\cup\lbrack x,x^{\prime}]$, whence
\[
P_{v}(\Omega(x))\subset P_{v}(\Omega(x^{\prime})).
\]
In particular,
\begin{equation}
\mathcal{H}^{1}(P_{v}(\Omega(x)))\leq\mathcal{H}^{1}(P_{v}(\Omega(x^{\prime
})))\ \hbox{ for }v\in\mathbb{S}^{n-1}, \label{es-1}
\end{equation}
Then
\eqref{polar} says that for all directions $v$ in $V$ we have
\[
\sup P_{v}(\Omega(x))\leq0<P_{v}(x^{\prime})\leq\sup P_{v}(\Omega(x^{\prime
})).
\]

Now for every $v\in
V$ we have
\begin{eqnarray}
\langle v,v_0\rangle= \langle \bar v , v_0\rangle+ \langle
v-\bar v,v_{0}\rangle \geq\tau -|v_{0}||v-\bar v| \geq \tau/2
\end{eqnarray}
because $\delta\leq \tau/2$.
This gives a lower bound for the length of the projected segment
$[x_{0},x^{\prime}]$ onto $\mathbb{R}v$. Precisely, recalling that 
\[
|x_{0}-x^{\prime}|=\frac{\lambda}{2}|x-x^{\prime}|,
\]

we get 

\[
P_{v}(x^{\prime})\geq \frac{\lambda \tau}{4}  |
x-x^{\prime} |  .
\]

This yields
\begin{equation}
\mathcal{H}^{1}(P_{v}(\Omega(x)))+\frac{\lambda \tau}{4}| x-x^{\prime}|
\leq\mathcal{H}^{1}(P_{v}(\Omega(x^{\prime})))\;\hbox{ for $v \in
V$.} \label{es-2}
\end{equation}
Integrating \eqref{es-2} for $v\in V$ and \eqref{es-1} for $v\in
\mathbb{S}^{1}\setminus V$, and summing up the resulting inequalities we
obtain \eqref{estim2}.

\bigskip

Consequently, our next goal is to determine $\delta>0$  and $\tau>0$ so that  \eqref{CL20} and \eqref{CL21} hold.

Here comes the point where the proof slightly differs  from  \cite[Theorem 3.3]{sc} and where the preliminary section about the curve bisector plays a role. First we notice that from the self-contracted property of the curve, the set $\Gamma(x)$ lies only on one side of the curve bisector $M(x,x')$ (the one containing $x$). On the other hand by applying Proposition \ref{prop2}, we know that $M(x,x')\subset S_{\kappa}(x,x')$ (see Proposition \ref{prop2} for the definition of $S_\kappa(x,x')$). Let $w$ be a vector parallel to $L_{v_0}$ and let  $\nu \in \mathbb{S}^1$ the  vector orthogonal to it pointing in the opposite direction with respect to $\Gamma(x)$. 

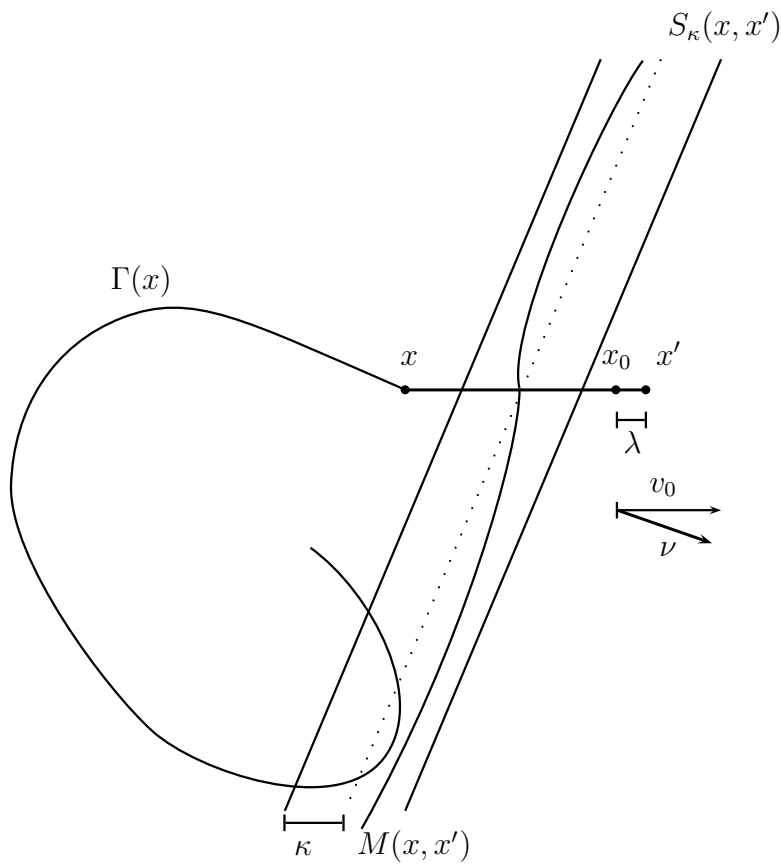
\begin{figure}[htbp]
\begin{center}
\scalebox{1} 
{
\begin{pspicture}(0,-5.6676955)(10.97566,5.6676955)
\psdots[dotsize=0.12](5.2737656,0.6492969)
\psdots[dotsize=0.12](8.473765,0.6492969)
\psline[linewidth=0.04cm](5.2737656,0.6492969)(8.473765,0.6492969)
\psline[linewidth=0.03cm,linestyle=dotted,dotsep=0.16cm](8.673765,5.049297)(4.473766,-4.950703)
\psline[linewidth=0.03cm](7.8737655,5.049297)(3.6737657,-4.950703)
\psline[linewidth=0.03cm](9.473765,5.049297)(5.2737656,-4.950703)
\usefont{T1}{ptm}{m}{n}
\rput(5.3252206,1.0742968){$x$}
\usefont{T1}{ptm}{m}{n}
\rput(8.775221,1.1342969){$x'$}
\usefont{T1}{ptm}{m}{n}
\rput(1.7652208,2.094297){$\Gamma(x)$}
\psdots[dotsize=0.12](8.073766,0.6492969)
\usefont{T1}{ptm}{m}{n}
\rput(8.08522,1.0542969){$x_0$}
\psline[linewidth=0.03cm,tbarsize=0.07055555cm 5.0,arrowsize=0.05291667cm 2.0,arrowlength=1.4,arrowinset=0.4]{|->}(8.073766,-0.95070314)(9.473765,-0.95070314)
\usefont{T1}{ptm}{m}{n}
\rput(8.705221,-0.64570314){$v_0$}
\psline[linewidth=0.03cm,tbarsize=0.07055555cm 5.0]{|-|*}(3.6537657,-5.110703)(4.453766,-5.110703)
\usefont{T1}{ptm}{m}{n}
\rput(3.9152207,-5.445703){$\kappa$}
\psline[linewidth=0.03cm,tbarsize=0.07055555cm 5.0]{|-|*}(8.073766,0.24929687)(8.473765,0.24929687)
\usefont{T1}{ptm}{m}{n}
\rput(8.29522,-0.045703124){$\lambda$}
\usefont{T1}{ptm}{m}{n}
\rput(9.525221,5.474297){$S_{\kappa}(x,x')$}
\psbezier[linewidth=0.03](6.7937655,0.6492969)(6.7937655,-0.15070313)(6.033766,-2.810703)(4.6937656,-5.190703)
\psbezier[linewidth=0.03](6.7937655,0.66929686)(6.5937657,1.4892969)(8.013765,4.489297)(8.433765,5.029297)
\usefont{T1}{ptm}{m}{n}
\rput(5.3952208,-5.425703){$M(x,x')$}
\psbezier[linewidth=0.03](5.2337656,0.66929686)(3.1137657,1.5892968)(2.4469426,1.9687141)(1.5137657,1.6092969)(0.5805887,1.2498796)(0.06753132,0.38872665)(0.03376566,-0.6107031)(0.0,-1.6101329)(1.3137656,-3.3507032)(1.9137657,-3.8907032)(2.5137656,-4.430703)(4.3425546,-5.0463142)(4.973766,-4.2707033)(5.6049767,-3.495092)(4.8268657,-2.032827)(4.013766,-1.4507031)
\psline[linewidth=0.04cm,arrowsize=0.05291667cm 2.0,arrowlength=1.4,arrowinset=0.4]{->}(8.093765,-0.95070314)(9.3537655,-1.3907031)
\usefont{T1}{ptm}{m}{n}
\rput(8.765221,-1.4657031){$\nu$}
\end{pspicture} 
}
\end{center}
\caption{$\Gamma(x)$ lies on one side of the curve bisector $M(x,x')$ which is contained in  the strip $S_\kappa(x,x')$. } \label{picture1}
\end{figure}

Then  it holds, thanks to our definition of $x_0$ in correspondance to the width of $S_{\kappa}$,
\begin{eqnarray}
\langle  y-x_0, \nu \rangle\leq -\frac{\lambda}{2}\leq 0 \quad \quad \forall y \in \Gamma(x). \label{tail1}
\end{eqnarray}

Moreover, the angle between $\nu$ and $v_0$ can be estimated by the constant $\alpha_0$ defined in Remark \ref{defalpha0}, which yields
$$\langle \nu, v_0\rangle \geq \cos(\frac{\pi}{2}-\alpha_0)=:\tau_1>0.$$

The vector $\nu$ is very close to satisfy the required conditions \eqref{CL20} and \eqref{CL21}. Actually, as already used in \cite{sc}, the desired vector $\bar v$   will be  constructed as a little perturbation of $\nu$, precisely, of the form
$$\nu_\varepsilon:=\frac{\nu+\varepsilon \nu^\bot}{|\nu+\varepsilon \nu^\bot|}$$
 for some $\varepsilon>0$ small, to be chosen later. 
 
 Let us check first that for any $\varepsilon\leq \tau_1/6$, it holds
 $$\langle \nu_\varepsilon, v_0\rangle\geq \tau_1/2$$ 
 so that condition \eqref{CL20} would be satisfied for any of those $\nu_\varepsilon$, with $\tau=\tau_1/2$.  Indeed, using the inequality $\sqrt{1+t}\leq 1+\frac{t}{2}$,  we obtain
 \begin{eqnarray}
|\nu-\nu_\varepsilon|^2=\frac{1}{1+\varepsilon^2}[(\sqrt{1+\varepsilon^2}-1)^2+\varepsilon^2]\leq  \frac{1}{2}\varepsilon^4+\varepsilon^2\notag 
\end{eqnarray}
which yields, for $\varepsilon\leq 1$, 
 \begin{eqnarray}
|\nu-\nu_\varepsilon|\leq 3\varepsilon.\label{estimationeps}
\end{eqnarray}
Consequently,
 $$\langle \nu_\varepsilon, v_0\rangle=\langle \nu, v_0\rangle+\langle \nu_\varepsilon-\nu, v_0\rangle\geq\tau_1-|\nu-\nu_\varepsilon|\geq \tau_1/2$$
 when $3\varepsilon\leq \tau_1/2$.

 Now it remains to find $\varepsilon\leq \tau_1/6$ that would moreover make $\nu_\varepsilon$ satisfying an inequality like  \eqref{CL21}. To do so, we decompose the half space $H:=\{y \; : \; P_{\nu}(y)\leq 0\}$ in two different regions. Let
$$H_1:=\{y \; : \; \langle \xi_0(y), \nu \rangle\leq -2\mu\}$$
and
$$H_2:=\{y \; : \;-2\mu< \langle \xi_0(y), \nu \rangle\leq 0\},$$
for some $\mu>0$ that will be chosen later. Notice  that, by \eqref{estimationeps}, if  $\varepsilon\leq \min(\tau_1/6,\mu/3)$, then
 \begin{equation}
\text{for all }y \in  H_1, \quad \quad \langle \xi_0(y) ,\nu_\varepsilon \rangle\leq \langle \xi_0(y) ,\nu \rangle + |\nu-\nu_\varepsilon| \leq -\mu, \label{firsthyp}
\end{equation}
as desired, and we only need to take care of  $H_2$. For this purpose, we furthermore decompose $H_2$ itself in the two regions, 
$$H_\mu:=H_2\cap B(x_0,1)$$
$$H'_\mu:=H_2\setminus H_\mu.$$
We know from  \eqref{tail1} that for all $y\in H_\mu$,
$$\langle \xi_0(y) ,\nu\rangle=\langle  \frac{y-x_0}{|y-x_0|}, \nu \rangle\leq -\frac{\lambda}{2|y-x_0|}\leq -\frac{\lambda}{2} .$$
Therefore, if $\varepsilon\leq \min(\alpha_0/6,\mu/3, \lambda/12)$ then \eqref{firsthyp} holds and moreover,
$$ \langle \xi_0(y) ,\nu_\varepsilon \rangle\leq \langle \xi_0(y) ,\nu \rangle + |\nu-\nu_\varepsilon| \leq -\frac{\lambda}{4},  \quad \text{for all } y \in H_\mu.$$
It remains  finally to consider $y\in H'_\mu$. For this we may assume that $2\mu\leq \frac{1}{2}$ so that $H'_\mu$ is divided in two connected components, one denoted $H^+_\mu$ lying in the upper half-space $\{x \; : \;\langle x,v_0^\bot\rangle \geq 0\}$ and another one denoted $H^-_\mu$ lying in the other part $\{x \; : \;\langle x,v_0^\bot\rangle \leq 0\}$.  Now for $\mu\in (0,\frac{1}{2})$ let us define the quantity
$$\beta_\mu:=\sup \{\langle \xi_0(x), \xi_0(y)\rangle \; : \; x\in H^+_\mu \text{ and } y \in H^-_\mu \}.$$
It is clear that $\beta_\mu\to -1$ when $\mu$ goes to $0$. Therefore, since $\alpha_0>0$, there exists $\mu>0$ small enough (depending only on  $\alpha_0$), such that $\beta_{\mu}<-\cos(\alpha_0)$. By applying  Lemma \ref{estimateL}, we infer that the following alternative holds
$$\Gamma(x)\cap H'_\mu\subset H_\mu^+\quad \text{ or }\quad \Gamma(x)\cap H'_\mu \subset H_\mu^-.$$
Next, we set  
$$ \varepsilon_0= \min(\alpha_0/6,\mu/3, \lambda/12).
$$
In the first case (i.e. when $\Gamma(x)\cap H'_\mu \subset H_\mu^+$) we define
$$\bar \nu:=\nu_{-\varepsilon_0} $$
and in the second case (i.e. when $\Gamma(x)\cap H'_\mu\subset H_\mu^-$) we define
$$\bar \nu:=\nu_{\varepsilon_0}.$$
We finally check that this choice suit our purposes. Indeed,  it is not difficult to see  that 
$$\sup_{y \in H_\mu^+}\Big\langle \xi_0(y) , \frac{v_0-\varepsilon_0 v_0^\bot}{|v_0-\varepsilon_0 v_0^\bot|} \Big\rangle=\cos\big(\frac{\pi}{2}+ \theta\big),$$
where $\theta$  is the  angle between $\nu$ and $\nu_{-\varepsilon_0}$ (see Figure \ref{picture2}).

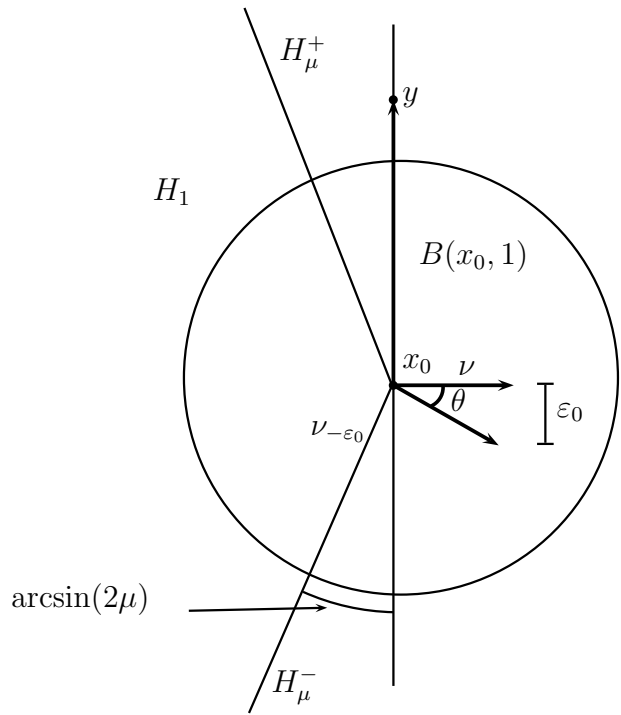
\begin{figure}[htbp]
\begin{center}
{
\begin{pspicture}(0,-4.6957283)(8.08,4.6957283)
\psline[linecolor=black, linewidth=0.03](5.08,4.470266)(5.08,-4.3297343)
\psdots[linecolor=black, dotsize=0.12](5.08,-0.32973424)
\psline[linecolor=black, linewidth=0.05, arrowsize=0.05291667cm 2.0,arrowlength=1.4,arrowinset=0.4]{->}(5.08,-0.32973424)(6.68,-0.32973424)
\psline[linecolor=black, linewidth=0.05, arrowsize=0.05291667cm 2.0,arrowlength=1.4,arrowinset=0.4]{->}(5.08,-0.32973424)(6.48,-1.1297343)
\psline[linecolor=black, linewidth=0.03, tbarsize=0.07055555cm 5.0]{|*-|*}(7.1,-0.30973426)(7.1,-1.1097343)
\rput[bl](7.251111,-0.81195647){\normalsize{$\varepsilon_0$}}
\rput[bl](5.94,-0.18973425){\normalsize{$\nu$}}
\rput[bl](3.98,-1.0897342){\normalsize{$\nu_{-\varepsilon_0}$}}
\rput[bl](0.0,-3.3897343){\normalsize{$\arcsin(2\mu)$}}
\pscircle[linecolor=black, linewidth=0.03, dimen=outer](5.18,-0.22973424){2.9}
\rput[bl](5.42,1.1702658){\normalsize{$B(x_0,1)$}}
\rput[bl](5.1985188,-0.16232684){\normalsize{$x_0$}}
\rput[bl](3.56,3.8302658){\normalsize{$H^+_\mu$}}
\rput[bl](3.46,-4.569734){\normalsize{$H_\mu^-$}}
\psdots[linecolor=black, dotsize=0.12](5.08,3.4702659)
\rput[bl](5.2,3.3702657){\normalsize{$y$}}
\psline[linecolor=black, linewidth=0.05, arrowsize=0.05291667cm 2.0,arrowlength=1.4,arrowinset=0.4]{->}(5.08,-0.32973424)(5.08,3.4702659)
\rput[bl](1.88,2.0702658){\normalsize{$H_1$}}
\rput[bl](5.8414817,-0.6749194){\normalsize{$\theta$}}
\psarc[linecolor=black, linewidth=0.048, dimen=outer](5.47,-0.35973424){0.27}{-65.85446}{3.0127876}
\psline[linecolor=black, linewidth=0.03](5.062222,-0.3326972)(3.1,4.6902657)
\psline[linecolor=black, linewidth=0.03](5.065185,-0.32528982)(3.16,-4.6897345)
\psarc[linecolor=black, linewidth=0.03, dimen=outer](5.15,-0.23973425){3.11}{-114.14554}{-91.20424}
\psline[linecolor=black, linewidth=0.03, arrowsize=0.05291667cm 2.0,arrowlength=1.4,arrowinset=0.4]{->}(2.36,-3.3297343)(4.2,-3.2897344)
\end{pspicture}
}
\end{center}
\caption{The minimum   angle between  $y-x_0$ and $\nu_{-\varepsilon_0}$ for $y \in H^+_{\mu}$ is achieved when $y-x_0 \in \nu^\bot$. } \label{picture2}
\end{figure}

This yields
$$\sup_{y \in H_\mu^+}\Big\langle \xi_0(y) , \frac{v_0-\varepsilon_0 v_0^\bot}{|v_0-\varepsilon_0 v_0^\bot|} \Big\rangle=-\frac{\varepsilon_0}{\sqrt{1+\varepsilon_0}} \leq - \frac{\varepsilon_0}{2},$$
provided that $\varepsilon_0\leq 1$.

Gathering all the estimates together, we have found some $\bar \nu$ satisfying  
$$\sup_{y\in \Gamma(x)}  \langle  \xi_0(y) ,  \bar \nu \rangle\leq-\tau,$$
with $\tau =\min(\mu, \frac{\lambda}{4},\frac{\varepsilon_0}{2})$, and this finishes the proof of the Theorem.
\end{proof}


\begin{thebibliography}{99}




\bibitem{BDLM}\textsc{J. Bolte, A. Daniilidis, O. Ley, L. Mazet},%
Characterizations of \L ojasiewicz inequalities: subgradient flows, talweg,
convexity, \emph{Trans. Amer. Math. Soc.} \textbf{362} (2010), 3319--3363.




\bibitem{sc}\textsc{A. Daniilidis, G. David, E. Durand-Cartagena and A. Lemenant} Rectifiability of Self-contracted curves in the euclidean space and applications.  \textit{J. Geom. Anal. }
\textbf{25} (2015), no. 2, 1211--1239. 



\bibitem{scR}\textsc{A. Daniilidis, R. Deville, E. Durand-Cartagena and L. Rifford} Self contracted curves in Riemannian manifolds  \textit{Preprint. }
(2015)




\bibitem{mY}\textsc{A. Daniilidis, Y. Garcia Ramos}, Some remarks on the
class of continuous (semi-)strictly quasiconvex functions, \emph{J. Optim.
Theory Appl.} \textbf{133} (2007), 37--48.

\bibitem{DLS}\textsc{A. Daniilidis, O. Ley, S. Sabourau}, Asymptotic
behaviour of self-contracted planar curves and gradient orbits of convex
functions, \emph{J. Math. Pures Appl.} \textbf{94} (2010), 183--199.





\bibitem{Loja63}\textsc{S. \L ojasiewicz}, ``Une
propri\'{e}t\'{e} topologique des sous-ensembles analytiques
r\'{e}els.'', in: \textit{Les \'{E}quations aux D\'{e}riv\'{e}es
Partielles}, pp. 87--89, \'{E}ditions du centre National de la
Recherche Scientifique, Paris, 1963.


\bibitem{MP}\textsc{P. Manselli, C. Pucci}, Maximum length of steepest
descent curves for quasi-convex functions, \emph{Geom. Dedicata} \textbf{38}
(1991), 211--227.

\bibitem{PalDem82}\textsc{J. Palis, W.  \& De Melo}, \emph{Geometric theory
of dynamical systems. An introduction,} (Translated from the Portuguese by A.
K. Manning), Springer-Verlag, New York-Berlin, 1982.

\end{thebibliography}
\end{document}